\documentclass[11pt]{article}
\usepackage[fleqn]{amsmath}
\usepackage{amssymb, amsthm, latexsym, amsfonts, epsfig,  graphicx,epstopdf,subcaption,caption,mathrsfs}
\usepackage{float}
\usepackage{ulem}
\usepackage{enumitem}
 \usepackage{color}
\usepackage{amsmath}
\numberwithin{equation}{section}
\def\R{\mathbb{R}}

\def\P{\mathbb{P}}
\def\E{\mathbb{E}}
\def\N{\mathbb{N}}

\def\kasten{$~~\mbox{\hfil\vrule height6pt width5pt depth-1pt}$ }

\usepackage{geometry}
\newtheorem{Theorem}{Theorem}[section]

\newtheorem{Definition}[Theorem]{Definition}

\newtheorem{Lemma}[Theorem]{Lemma}

\begin{document}

\def\kasten{$~~\mbox{\hfil\vrule height6pt width5pt depth-1pt}$ }
\makeatletter\def\theequation{\arabic{section}.\arabic{equation}}
\makeatother
\newtheorem{theorem}{Theorem}

\centerline{\large \bf A parameter estimator based on Smoluchowski-Kramers approximation
 \footnote{This work was partly supported by the NSF grant 1620449 and the NSFC grants 11531006 and 11771449.}\hspace{2mm}
\vspace{1mm}\vspace{1mm}\\ }
\smallskip
\centerline{\bf Ziying He$^{a,}\footnote{ziyinghe@hust.edu.cn}$, 
\bf Jinqiao Duan$^{d,}\footnote{Corresponding author, duan@iit.edu}$
\bf Xiujun Cheng$^{a,}$\footnote{xiujuncheng@hust.edu.cn}}
\smallskip
 \centerline{${}^a$ Center for Mathematical Sciences,}
 \centerline{$\&$ School of Mathematics and Statistics,}
 \centerline{$\&$ Hubei Key Laboratory for Engineering Modeling and Scientific Computing}
\centerline{Huazhong University of Sciences and Technology, Wuhan 430074,  China}
\centerline{${}^d$Department of Applied Mathematics, Illinois Institute of Technology, Chicago, IL 60616, USA}


\begin{abstract}
We devise a simplified parameter estimator for a second order stochastic differential equation by a first order system based on the Smoluchowski-Kramers approximation. We establish the consistency of the estimator by using $\Gamma$-convergence theory. We further illustrate our estimation method by an experimentally studied movement model of a colloidal particle immersed in water under conservative force and constant diffusion.
\end{abstract}

{\footnotesize }


\section{Introduction}
We estimate an unknown system parameter in the drift term of a stochastic system satisfying Newton's law by  Smoluchowski-Kramers approximation. The Smoluchowski-Kramers approximation was introduced by Kramers \cite{Kramers H} and Smoluchowski \cite{Smoluchowski M} to compute a small parameter limit for irregularly moving particles. It has been extended to more general systems  \cite{Mark Freidlin}.
The Smoluchowski-Kramers approximation is the main justification for replacement of the second order movement equation by the first order force-distance relation equation in physical experiments. It has been used to devise a filtering method to cut the dimension by half and bypass the dependence on velocity observation 
\cite{Andrew Papanicolaou}.

In this Letter, we devise a parameter estimator for a second order stochastic system by its first order Smoluchowski-Kramers approximation system.  We verify the consistency property of the estimator by employing $\Gamma$-convergence theory, which was developed to study the optimization problems \cite{Gianni Dal Maso}. This $\Gamma$-convergence theory is about a family of functionals' convergence to a limit functional, and the corresponding convergence for the minimizers. In general, this latter convergence can not be assured. However, in our case, the objective function is verified to be $\Gamma$-convergent, and the corresponding minimizer converges to the true parameter value. Thus the consistency about the estimator holds.
We apply our parameter estimation method to determine the force in an experimental system. The force can be measured by modeling the particle with small mass through the force-distance relationship \cite{Giovanni Volpe}.

\section{Settings and assumptions}
We consider the parameter estimation on $\vartheta$ in the parameter space $\Theta\subset\R$ in the following second order Newton equation of motion under random fluctuations
\begin{eqnarray}\label{eq1}
\mu \ddot{x}^{\mu}(t)=b(x^{\mu}(t),\vartheta)-\gamma\dot{x}^{\mu}(t)+\sigma\dot{W}_{t}, \qquad x^{\mu}(0)=x_{0}\in \R^{n},\quad \dot{x}^{\mu}(0)=v_{0}\in\R^{n},
\end{eqnarray}
where $W_{t}$ is an $n$-dimensional Brownian motion \cite{JinqiaoDuan} defined on sample space $\Omega$ with probability $\P$ , $\gamma$ is the friction coefficient (a positive constant), $\sigma$ is a constant, and $b$ is a continuous drift function  from $\R^{n}\times\Theta$ to $\R^{n}$.\\
The second order equation \eqref{eq1} can be rewrite as a first order system
\begin{eqnarray}\label{eq2}
\left\{\begin{array}{l}
\dot{x}^{\mu}(t)=v^{\mu}(t),\qquad\qquad\qquad\qquad\quad\qquad\,\,\,\,\,\, x^{\mu}(0)=x_{0}\in\R^{n},\\
\dot{v}^{\mu}(t)=\frac{1}{\mu}b(x^{\mu}(t),\vartheta)-\frac{\gamma}{\mu}v^{\mu}(t)+\frac{\sigma}{\mu}\dot{W}_{t},\qquad v^{\mu}(0)=v_{0}\in\R^{n}.
\end{array}
\right.
\end{eqnarray}
The Smoluchowski-Kramers approximation assures that equation \eqref{eq1} converges in some sense to the following Smoluchowski equation \cite{Mark Freidlin}, as $\mu\to 0$,
\begin{eqnarray}\label{eq3}
\dot{x}(t)=\frac{1}{\gamma}b(x(t),\vartheta)+\frac{\sigma}{\gamma}\dot{W}_{t},\qquad x(0)=x_{0}\in\R^{n}.
\end{eqnarray}
By revising the proof in \cite{Edward Nelson} (page $59$), we have
\begin{eqnarray}\label{lim}
\lim_{\mu\to0}x^{\mu}(t)=x(t)\qquad \mbox{a.s.,}
\end{eqnarray}
 for all $v_{0}\in\R^{n}$, uniformly for $t$ in compact subintervals of $[0,\infty)$.

Suppose we have available the observations $\{x^{\mu}(t_{k})\}_{k=1}^{n}$ on time interval $[0,T]$ with partition $\Delta_{n}=\{0=t_{1}<\cdots<t_{k}<\cdots<t_{n}=T\}$, $\Delta_{k}t=t_{k+1}-t_{k}$ $(k=1,2,\cdots,n-1)$ and $T=\Delta n^{\frac{1}{2}}$ with a positive constant $\Delta$ \cite{Kasonga}. We can estimate $\vartheta$ by the first order $n$-dimensional system \eqref{eq3}. We devise a least square estimator for $\vartheta$ in equation \eqref{eq1} by equation \eqref{eq3}. \\
Define the objective function
\begin{eqnarray}\label{func1}
F_{n}^{\mu}(\vartheta)=\sum_{k=1}^{n}\frac{||x^{\mu}(t_{k})-x^{\mu}(t_{k-1})-\frac{1}{\gamma}b(x^{\mu}(t_{k-1}),\vartheta)\Delta_{k}t||^{2}}{\Delta_{k}t}.
\end{eqnarray}
Then, as $\mu\to0$, this objective function will tend almost surely to the following function
\begin{eqnarray}\label{func2}
F_{n}(\vartheta)=\sum_{k=1}^{n}\frac{||x(t_{k})-x(t_{k-1})-\frac{1}{\gamma}b(x(t_{k-1}),\vartheta)\Delta_{k}t||^{2}}{\Delta_{k}t}.
\end{eqnarray}
This is the objective function of the least square estimation for equation \eqref{eq3} with observations $\{x(t_{k})\}_{k=1}^{n}$. We denote the true parameter value as $\vartheta^{0}$, and the unique minimizer for $F_{n}(\vartheta)$ as $\vartheta_{n}$. Then under suitable conditions, this estimator possesses the consistency property \cite{Kasonga}
$$\vartheta_{n}\to\vartheta^{0} \quad \mbox{a.s.} \quad \mbox{as } n\to\infty.$$
{\bf Let $\vartheta_{n}^{\mu}$ minimize $F_{n}^{\mu}(\vartheta)$. We want to infer the consistency property
$$\vartheta_{n}^{\mu}\to\vartheta^{0} \quad \mbox{a.s.} \quad \mbox{as } \mu\to0,\, n\to\infty.$$}
To this end, we make the following assumptions as in  \cite{Kasonga}:
\begin{itemize}
\item [A1.] The parameter space $\Theta$ is a compact subspace of $\R$.
\item [A2.] The drift term $b(x,\vartheta)$ is Lipschitz in both variables, that is, there exist continuous functions $K(\cdot)$ and $C(\cdot)$ such that
$$||b(x,\vartheta)-b(y,\vartheta)||_{\R^{n}}\leq|K(\vartheta)|\cdot||x-y||_{\R^{n}},$$
$$||b(x,\vartheta)-b(x,\varphi)||_{\R^{n}}\leq|C(x)|\cdot|\vartheta-\varphi|,$$
for all $\vartheta, \varphi\in\Theta$, $x,v\in \R^{n}$. Here
$$\sup_{\vartheta\in\Theta}|K(\vartheta)|=K<\infty,\mbox{ and } \E|C(x_{0})|^{m}=C_{m}<\infty\quad \mbox{for some } m>16,$$
with $\E$ the expectation with respect to $\P$.\\
The following growth condition is also satisfied for some constant $L$.
$$||b(x,\vartheta)||_{\R^{n}}^{2}\leq L(1+||x||_{\R^{n}}^{2}).$$
\item [A3.] The process $x_{t}=(x_{t},t\geq0)$ is stationary and ergodic with
$$\E||x_{0}||_{\R^{n}}^{m}<\infty\quad \mbox{for some } m>16.$$
\item [A4.] Recall that $\vartheta^{0}$ is the true parameter value in \eqref{eq1}. Then
$$\E||b(x_{0},\vartheta)-b(x_{0},\vartheta^{0})||_{\R^{n}}^{2}=0\quad\mbox{ iff }\quad \theta=\theta^{0}.$$
\end{itemize}
We need the following definitions to prove the consistency \cite{Gianni Dal Maso}.
\begin{Definition}
We say that a subset $K$ of a topological space $X$ is countably compact if every sequence in $K$ has at least a cluster point in $K$.
\end{Definition}
\begin{Definition}
We say that a function is coercive on topological space $X$, if the closure of the set $\{F\leq a\}$ is countably compact in $X$ for every $a\in\R$.
\end{Definition}
\begin{Definition}
We say that the sequence $(F_{h})$ is equi-coercive for $h\in\N$ (the natural number set) on topological space $X$, if for every $a\in\R$ there exists a closed countably compact subset $K_{a}$ of $X$ such that $\{F_{h}\leq a\}\subset K_{a}$ for every $h\in\N$.
\end{Definition}

\section{Consistency}
We first prove the $\Gamma$-convergence for the objective function $F_{n}^{\mu_{k}}(\vartheta)$ with respect to every sequence $(\mu_{k})$ tending to $0$. We denote the extended real line as $\overline{\R}=\R\cup\{\pm\infty\}$. Recall the definition about $\Gamma$-convergence \cite{Gianni Dal Maso}.
\begin{Definition}
The $\Gamma$-lower limit and the $\Gamma$-upper limit of the sequence $(F_{h})$ are the functions from $X$ into $\overline{\R}$ defined by
$$(\Gamma\mbox{-}\liminf_{h\to\infty}F_{h})(x)=\sup_{U\in\mathcal{N}(x)}\liminf_{h\to\infty}\inf_{y\in U} F_{h}(y),$$
$$(\Gamma\mbox{-}\limsup_{h\to\infty}F_{h})(x)=\sup_{U\in\mathcal{N}(x)}\limsup_{h\to\infty}\inf_{y\in U} F_{h}(y),$$
where $U$ is an open subset in $X$ and $\mathcal{N}(x)$ denotes the set of all open neighbourhoods of $x$ in $X$.\\
If there exists a function $F:X\to\overline{\R}$ such that $\Gamma\mbox{-}\liminf_{h\to\infty}F_{h}=\Gamma\mbox{-}\limsup_{h\to\infty}F_{h}=F$, then we write $F=\Gamma\mbox{-}\lim_{h\to\infty}F_{h}$ and we say that the sequence $F_{h}$ $\Gamma$-converges to $F$ (in $X$) or that $F$ is the $\Gamma$-limit of $(F_{h})$ (in $X$).
\end{Definition}
\begin{Lemma}\label{pp1}
Under the assumption A2, and for $n\in\N$,
$F_{n}^{\mu}(\vartheta)$ {\bf $\Gamma$-converges} to $F_{n}(\vartheta)$ a.s., as $\mu\to0$.
\end{Lemma}
\begin{proof}
For every $n\in\N$, the convergence \eqref{lim} holds uniformly on $[0,T]$. This implies that, for every $\varepsilon>0$, there exists a positive $\mu_{0}$, such that for $\mu\in(0,\mu_{0})$,
$$||x^{\mu}(t_{k})-x(t_{k})||<\varepsilon \qquad \mbox{a.s.,}\quad \mbox{ for } k=1,2,\dots,n.$$
Hence for every $n\in\N$, and the preceding $\mu$,
\begin{eqnarray*}
\begin{aligned}
&\sup_{\vartheta\in\Theta}|F_{n}^{\mu}(\vartheta)-F_{n}(\vartheta)|\\
\leq&\sup_{\vartheta\in\Theta}\bigg\{\sum_{k=1}^{n}\frac{1}{\Delta_{k}t}\bigg[||x^{\mu}(t_{k})-x(t_{k})||+||x^{\mu}(t_{k-1})-x(t_{k-1})||+\frac{\Delta_{k}t}{\gamma}||b(x^{\mu}(t_{k-1}),\vartheta)-b(x(t_{k-1}),\vartheta)||\bigg]^{2}\bigg\}\\
\leq&\sup_{\vartheta\in\Theta}\bigg\{\sum_{k=1}^{n}\frac{1}{\Delta_{k}t}(\varepsilon+\varepsilon+\frac{\Delta_{k}t}{\gamma}|K(\vartheta)|\varepsilon)\bigg\}\\
\leq&M_{n}\varepsilon,
\end{aligned}
\end{eqnarray*}
for a positive constant $M_{n}$ by the assumption A2. Thus  for every $n\in\N$,
$$F_{n}^{\mu}(\vartheta)\to F_{n}(\vartheta)\quad \mbox{a.s.} \mbox{ as }\mu\to0,\mbox{ uniformly for } \vartheta\in\Theta.$$
Note that the drift $b(x,\vartheta)$ is a continuous function of $\vartheta$, so is $F_{n}(\vartheta)$. Hence $F_{n}(\vartheta)$ is a lower semi-continuous function on $\vartheta$.  According to Proposition $5.2$ in \cite{Gianni Dal Maso}, we infer the result.
\end{proof}
For the corresponding minimizers' convergence, we still need the equi-coercive condition for the objective functions  $F_{n}^{\mu_{k}}(\vartheta)$.
\begin{Lemma}\label{pp2}
For every $n\in\N$, $F_{n}^{\mu_{k}}(\vartheta)$ is {\bf equi-coercive} with respect to $k\in\N$ for every sequence $(\mu_{k})$ satisfying $\mu_{k}\to0$.
\end{Lemma}
\begin{proof}
For every $\varepsilon>0$, there exists a $\mu_{0}>0$ such that, for $\mu_{k}\in(0,\mu_{0})$,
$$F^{\mu_{k}}_{n}(\vartheta)>F_{n}(\vartheta)-\varepsilon.$$
The function $F_{n}(\vartheta)-\varepsilon$ is continuous for $\vartheta$ from $\Theta\subset\R$ to $\overline{\R}$. Therefore $\{F_{n}(\vartheta)-\varepsilon\leq a\}$ is a closed set in space $\Theta$ for any $a\in\R$, further it is also a countably compact set. That is, $F_{n}(\vartheta)$ is a coercive function and a lower semi-continuous function for $\vartheta$. According to Proposition $7.7$ in \cite{Gianni Dal Maso}, we infer that the sequence $(F^{\mu_{k}}_{n}(\vartheta))$ is equi-coercive.
\end{proof}
\begin{Theorem}\label{th}
Under the assumptions A1, A2, A3 and A4, the {\bf consistency} property holds, i.e.
$$\vartheta_{n}^{\mu}\to\vartheta^{0} \quad\mbox{a.s.} \mbox{ as } \mu\to0,\, n\to\infty.$$
\end{Theorem}
\begin{proof}
By Corollary $7.24$ in \cite{Gianni Dal Maso},  together with Lemma \ref{pp1} and \ref{pp2}, we conclude that for every $n\in\N$ and every sequence $(\mu_{k})$,
$\vartheta_{n}^{\mu_{k}}\to\vartheta_{n} \mbox{ a.s.} \mbox{ as } \mu_{k}\to0.$
Hence for every $n\in\N$,
$$\vartheta_{n}^{\mu}\to\vartheta_{n} \quad\mbox{a.s.} \mbox{ as } \mu\to0.$$
Noting the following consistency from \cite{Kasonga},
$$\vartheta_{n}\to\vartheta^{0}\quad\mbox{a.s.} \mbox{ as } n\to+\infty,$$
we infer  the  consistency property.
\end{proof}
Hence for sufficiently large $n$, we can adjust the system parameter $\mu$ sufficiently small and estimate the value of $\vartheta^{0}$ in equation \eqref{eq1} by the observations $\{x^{\mu}(t_{k})\}_{k=1}^{n}$ and the equation \eqref{eq3}.

\section{An example}
The measurement about force is usually needed in various scientific systems \cite{Giovanni Volpe, Hottovy S, Thomas Brettschneider, Prashant Sinha}.
In an experimental system of a colloidal particle \cite{Giovanni Volpe}, the Smoluchowski-Kramers approximation provides a way to measure the force of the system \cite{Hottovy S1}.
 The motion of a particle with mass $\mu$ in the presence of thermal noise is governed by the following Newton's law \cite{Hottovy S1},
\begin{eqnarray}\label{ex1}
\left\{\begin{array}{l}
\dot{x}^{\mu}(t)=v^{\mu}(t),\quad\quad\quad\quad\quad\quad\quad\quad\quad\quad\quad\quad\quad\quad\quad\,\,\,\,\,\,\,\quad x^{\mu}(0)=0,\\
\mu\dot{v}^{\mu}(t)=F(x^{\mu}(t))-\frac{k_{B}T_{tem}}{D(x^{\mu}(t))}v^{\mu}(t)+\frac{\sqrt{2}k_{B}T_{tem}}{\sqrt{D(x^{\mu}(t))}}\dot{W}_{t},\qquad v^{\mu}(0)=v_{0}.\\
\end{array}
\right.
\end{eqnarray}
where $v$ is the velocity, $x$ is the position of the particle, $k_{B}$ is the Boltzmann constant, $T_{tem}$ is the absolute temperature of the fluid, $D(x)$ is a hydrodynamical diffusion gradient, and $\gamma=\frac{k_{B}T_{tem}}{D}$ is the friction coefficient satisfing the  fluctuation-dissipation relation $\sigma=\sqrt{2k_{B} T_{tem}\gamma}$. The measurement for  the velocity is difficult to obtain due to the irregular motion of the particle. But researchers can implement measurement about the position of trajectory of particle's motion. For example, a single particle evanescent light scattering technique known as total internal reflection microscopy has been used \cite{Giovanni Volpe}. The particle's trajectory is obtained from the scattering intensities, which depend on its position relative to the interface.  This measurement problem calls for a relationship between position and force instead of velocity and force. The Smoluchowski-Kramers approximation provides a theoretical basis for this \cite{Mark Freidlin}. Researchers can measure the force with the help of the movement of small mass particle, through the relation between position and force,
\begin{eqnarray}\label{ex2}
\dot{x}(t)=\frac{F(x(t))D(x(t))}{k_{B}T_{tem}}+D^{'}(x(t))+\sqrt{2D(x(t))}\dot{W}_{t},\qquad x(0)=0.
\end{eqnarray}
Volpe et.al \cite{Giovanni Volpe} experimentally studied a colloidal particle immersed in water and diffusing in a closed sample cell above a planar wall, placed at $x=0$. The conservative force acting on the particle is
$$F(x,\vartheta)=\vartheta e^{-\kappa x}-G_{eff},$$
where $\vartheta$ is a prefactor depending on the surface charge densities, and $\kappa^{-1}=18$ nm denotes the Debye length.  The term $G_{eff}=4/3\pi R^{3}(\rho_{p}-\rho_{s})g$
accounts for the effective gravitational contribution, where the diameter of the colloidal particle is $2R=1.31\pm0.04$ $\mu$m,
the density of the particle is $\rho_{p}=1.51$ g/cm$^{3}$, the density of water is $\rho_{s}=1.00$ g/cm$^{3}$, and the gravitational acceleration constant is $g$. The particle's trajectory perpendicular to the wall is sampled with nanometer resolution. Under constant diffusion, $\gamma$ and $\sigma$ are constant quantities. The movement equation is depicted by system \eqref{eq2},
\begin{eqnarray}\label{ex3}
\left\{\begin{array}{l}
\dot{x}^{\mu}(t)=v^{\mu}(t),\quad\quad\quad\quad\quad\quad\quad\quad\quad\quad\quad\quad\,\, x^{\mu}(0)=0,\\
\mu\dot{v}^{\mu}(t)=F(x^{\mu}(t),\vartheta)-\gamma v^{\mu}(t)+\sigma\dot{W}_{t},\qquad v^{\mu}(0)=v_{0}.\\
\end{array}
\right.
\end{eqnarray}
The mass of the particle is taken sufficiently small such that
\begin{eqnarray}\label{ex4}
\dot{x}(t)=\frac{1}{\gamma} F(x(t),\vartheta)+\frac{\sigma}{\gamma}\dot{W}_{t},\qquad x(0)=0.
\end{eqnarray}
Thus the main result Theorem \ref{th} holds. We use the objective function \eqref{func1}, to estimate $\vartheta$. We make the unit of physical quantity unification as g, nm, and s. The force is
\begin{eqnarray}\label{force}
F(x,\vartheta)=\vartheta\cdot e^{-\frac{x}{18}}-\frac{4}{3}\pi\cdot(\frac{1.31}{2})^2\cdot0.51\cdot9.8\cdot10^{-3}.
\end{eqnarray}
The simulation result is shown in Figure \ref{fig}, which indicates that our estimator is a good approximation for the true parameter value.

\begin{figure}[H]
\centering
\includegraphics[width=2.in]{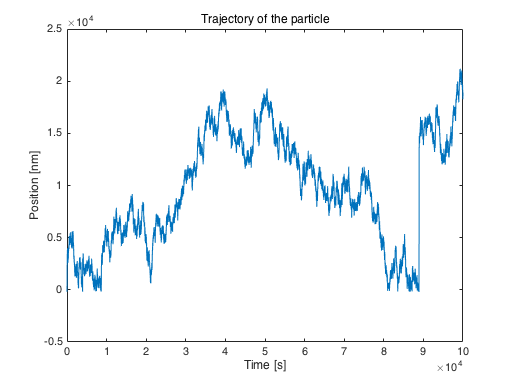}
\includegraphics[width=2.in]{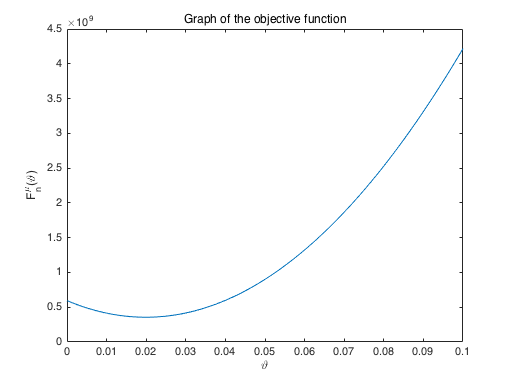}
\caption{Take $\gamma=\frac{1}{6}$, $\sigma=10$, $g=9.8$ m/s$^{2}$, $\theta^{0}=0.02$, $n=10^{5}$ and $\mu=0.001$:
(i)  The left subfigure is one trajectory of the particle under true parameter value $\vartheta^{0}$ simulated through Newton motion equation \eqref{ex3}.  (ii) The right subfigure is the graph of the objective function \eqref{func1} with drift force \eqref{force}. The simulation result for the parameter $\vartheta_{n}^{\mu}$ is $0.0200$.}\label{fig}
\end{figure}

\end{document}